\documentclass[12pt,reqno,a4paper]{amsart}

\usepackage{preamble}

\begin{document}

\title{A Structural Condition on Point Sets \\with Few Distinct Dot Products}

\author{Anshula Gandhi}  \address{Department of Pure Mathematics and Mathematical Statistics \\ Centre for Mathematical Sciences \\ University of Cambridge \\ Cambridge, UK} \email{ag2163@cam.ac.uk}

\begin{abstract}
    The distinct dot products problem, a variant of the Erdős distinct distances problem, asks ``Given a set $P_n$ of $n$ points in $\mathbb{R}^2$, what is the \emph{minimum} number $|D(P_n)|$ of distinct dot products they determine?'' The best proven lower bound is $|D(P_n)| = \Omega(n^{2/3+7/1425})$, due to work by Hanson--Roche‐Newton--Senger, and a recent improvement by Kokkinos. However, the best known construction determines $\Theta(n)$ dot products. We provide a structural condition that a point configuration $P_n$ would have to satisfy in order to have `few' dot products, by which we mean that $|D(P_n)| < n^{\frac{3}{4}(1-\epsilon)}$ for some $\epsilon > 0$. 
\end{abstract}
\maketitle

\vspace{-.75cm}

\section{Introduction} \label{sec:intro}

The distinct distances problem was first posed in 1946 by Erdős \cite{erdos1946sets}: given $n$ points in $\mathbb{R}^2$, what is the minimum number of distinct distances they can determine?  We are motivated by the following dot product variant of the problem, as discussed in \cite{steinerberger2010note, garibaldi2011erdos}. Given $n$ points in $\mathbb{R}^2$, what is the minimum number of distinct dot products they can determine? 

\begin{conjecture} 
    Suppose that $(P_n)_{n \in \mathbb{N}}$ is a sequence of point configurations, where each $P_n$ is a set of $n$ distinct points $\{p_1, \dots, p_n\}$ in $\mathbb{R}^2$. Let its set of dot products be given by $D(P_n):= \set{p_i \cdot p_j \mid  p_i, p_j \in P_n}$. Then $|D(P_n)| \gg n$.\footnote{We use $f(n) \ll g(n)$   to mean $f(n) = O(g(n))$.}
\end{conjecture}

Simple constructions show that the conjecture, if true, is sharp. For example, if $P_n$ consists of $n$ multiples of a single vector with the multiples in geometric progression, then $|D(P_n)|=2n-1$, and if $P_n$ consists of $n$ points equally spaced on a circle centered at the origin, then $|D(P_n)|=\lfloor n/2\rfloor+1$.  However, the best known lower bound is only $|D(P_n)| \gg n^{2/3+c}$ for a small constant $c$, so there is a sizeable gap \cite{Hanson2023, Kokkinos2025}.

Many variants of the dot products problem have been studied over the last decade: sampling points from a finite field or ring instead of the plane \cite{Covert2015, Blevins2021, Crosby2022}, restricting to counting dot products between successive pairs of points in a `chain' of points \cite{Kilmer2024}, bounding triples of points that determine a fixed pair of dot products \cite{Barker2015, Lund2015}, bounding distinct dot products in a dot product `tree' (where vertices are points in the configuration, and edge weights are given by dot products between points) \cite{Autry2022}, and a dot product variant of the Falconer distance problem \cite{bright2026pinned}. 

Our main result shows that any point set with $|D(P_n)| < n^{\frac{3}{4}(1-\epsilon)}$ (where $\epsilon > 0$) must satisfy a particular structural property which seems highly constraining.  Informally, this property is the existence of many rich lines with many closely-clustered points.  This result suggests that improving the lower bound to $|D(P_n)| \gg n^{3/4}$ might not be totally out of reach.

\subsection{The main result}\label{the-main-result}

To state our main result, we first formalize a particular notion of `closeness' of points along a line.

\begin{defn}[$b$-close]
    Let us call a pair $p,q$ of points along a radial line \emph{$b$-close} for some $b \in (0,1)$ if $|p|<|q|$ and $|p|/|q|$ falls within the interval $(b,1)$.
\end{defn}

Our main theorem identifies a property that must hold for any set $P$ with few dot products.


\begin{theorem}[Existence of many clustered lines] \label{many-clustered-lines}
    Fix $\epsilon \in (0,1]$ and $\delta \in [0,1]$. Let $P$ be a set of $n$ points in $\mathbb{R}^2$, with $n$ sufficiently large, where $|D(P)| < \left( \frac n2 \right)^{\frac34(1-\epsilon-\delta)}$. Then $P$ contains a subset $P'\subseteq P$ with $|P'|\ge  \frac n2 $ which can be partitioned into disjoint sets $P'=L_1\cup\cdots\cup L_k$, each contained in a radial line, such that $|L_i|\ge k_1 \left( \frac n2 \right)^{\frac12(1-\epsilon+3\delta)}$ and at least $\floor{(1-\left(\frac n2\right)^{-3\delta})|L_i|}$ consecutive pairs of points in $L_i$ are $(1-k_2 \left( \frac n2 \right)^{-2 \epsilon})$-close, where $k_1 = 2^{-2/3}$ and $k_2 = 2^{19/3}\pi^2$.
\end{theorem}

The theorem says that any configuration with very few dot products contains a large subset supported on rich radial lines, and that along each such line almost all consecutive gaps are very small (i.e. the ratio of their magnitudes is very close to 1).\footnote{Note that although most pairs of points are $b$-close, there may be gaps from one $b$-close pair to the next closest $b$-close pair, and as such, there may be multiple `clusters' of such points along such a line.}  Informally, $\epsilon$ determines how closely-clustered these points are, and $\delta$ determines how many consecutive pairs are closely-clustered.

Note this result describes various conditions on point sets $P$ with fewer than $n^{3/4}$ dot products.  For example, if we wish to study configurations with fewer than $n^{7/10}$ dot products, we may choose $\epsilon=1/30$ and $\delta=1/30$. In this case, the theorem guarantees the existence of many radial lines $L$ each containing at least $c_1 n^{8/15}$ points, where on each line, at least $(1-n^{-1/10})|L|$ pairs of consecutive points are $(1-c_2n^{-1/15})$-close for some $c_1$ and $c_2$.

\subsection{Outline of the proof}

The structure of the proof is roughly as follows.

We first outline some simple observations showing that if a point configuration has few distinct dot products, it must have both a rich line and a rich circle (this is done in Section \ref{preliminaries}).

We then show that if the points on the rich line are sufficiently well-spaced, the configuration determines many distinct dot products (this is done in Section \ref{line-and-circle-configurations}). Therefore, for a point configuration to have few dot products,  the rich line must contain many such clusters of points. Finally, we conclude that if every such set of points has one rich clustered line, then a large subset of the point set can be partitioned into rich clustered lines.

\section{Preliminaries}\label{preliminaries}

We first note the following quickly-derived facts.

\subsection{Existence of a rich line}\label{line-configurations}

\begin{lemma}[Bounding dot products on a line] \label{bounding-dot-products-on-line}
    If we have a set $L\subset \R^2$ consisting of $n$ points on a radial line, then they determine at least $n$ distinct dot products. 
\end{lemma} 

\begin{proof}     
    Rotating a point set does not change the number of dot products, so without loss of generality, we can assume the points are on the $x$-axis, with $x$-coordinates given by $\ell \subseteq \R$. Note that the size of the set of dot products $|L \cdot L|$ is the same as the size of the product set $|\ell \ell|$, and product sets in $\R$ have size at least $n$. 
\end{proof} 

Note that the above lemma is not sharp.\footnote{The bound could be improved to $2n-2$ if we assume $0 \notin L$.  Throughout this paper, as in previous work on the distinct dot products problem, we assume we are dealing with point sets $P$ such that $0 \notin P$.}  However, a lower bound of $n$ suffices for our purposes and helps simplify some of the following proofs.

We define the \textit{supporting lines} of a point configuration as follows:

\begin{defn}[Supporting lines]
    Given a set of points $P \subset \mathbb{R}^2$, let their \textit{supporting lines} be the set of radial lines which also intersect some point $p \in P$.
\end{defn}

\begin{figure}[H] \centering        
    \includegraphics[height=80px]{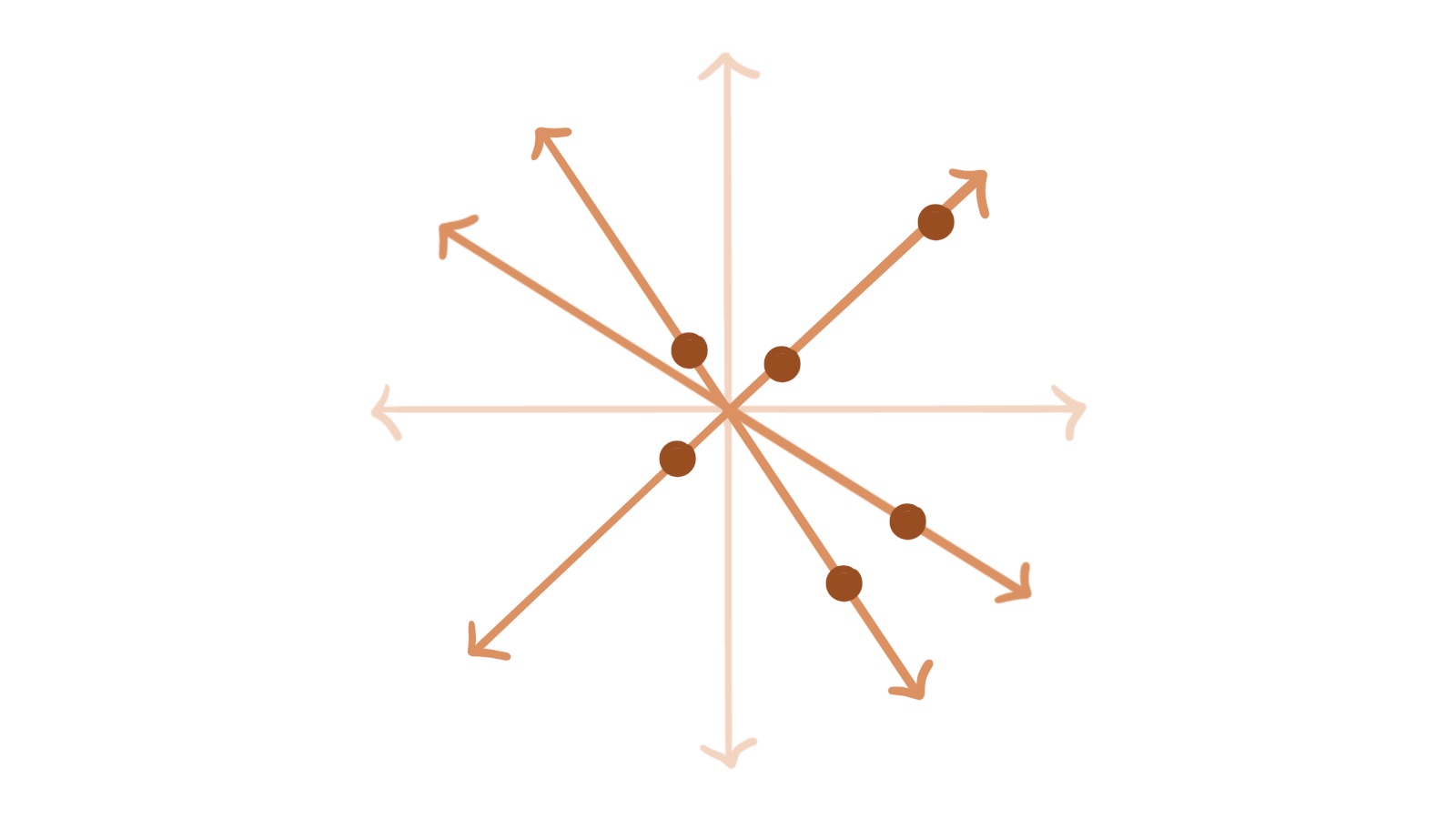} 
    \caption{A set of points and their supporting lines.} 
\end{figure}

First we note that a point configuration with few dot products must have many supporting lines.

\begin{lemma}[Lower bound on \# of supporting lines]  
    Let $P_n$ be a set of $n$ points in $\R^2$ with at most $m$ dot products. Then, $P_n$ has at least $n/m$ supporting lines. 
\end{lemma}

\begin{proof}
    By Lemma~\ref{bounding-dot-products-on-line}, each supporting line contains at most $m$ points, so the result follows.
\end{proof}

The next result shows that a point configuration with few dot products cannot have \emph{too many} supporting lines.

\begin{lemma}[Upper bound on \# of supporting lines, Lemma 3.1 in \cite{Hanson2023}]  \label{upper-bound-lines} 
    Let $P_n$ be a set of $n$ points in $\R^2$ with at most $m$ dot products. Then, $P_n$ has at most $64m^2/n$ supporting lines. 
\end{lemma}

\begin{proof} 
    This proof immediately follows from a lemma given in \cite{Hanson2023}, which uses the Szemerédi–Trotter theorem. Fix some configuration $P_n$, and let $L_n$ be its set of supporting lines. Lemma 3.1 from \cite{Hanson2023} says that there must exist some point $p \in P_n$ that determines many dot products with the other points in the set i.e.~$| \set{p \cdot q \mid q \in P_n} |\gg n^{1/2} |L_n|^{1/2}$.  An inspection of the proof shows that we have $| \set{p \cdot q \mid q \in P_n} | \geq \frac{1}{8} n^{1/2} |L_n|^{1/2}$. 

    It follows that:

    $$ m \geq |D(P_n)| \geq \frac{1}{8} n^{1/2} |L_n|^{1/2}.$$

    Rearranging, we have that $|L_n| \leq 64 \frac{m^2}{n}$. 
\end{proof}

The above result says that if a configuration has \emph{not too many} dot products, then it also has \emph{not too many} supporting lines. It follows that there exists one line which supports many points.

\begin{lemma}[Existence of a rich line] \label{rich-line}
    Let $P_n$ be a set of $n$ points in $\R^2$, with at most $m$ dot products. Then, $P_n$ must contain a supporting line with $n^2/(64m^2)$ points --- we call this the ``rich line.'' 
\end{lemma}

\begin{proof} 
    By  Lemma~\ref{upper-bound-lines}, there are at most $64 m^2 /n$ supporting lines. Therefore, at least one of these lines (what we call the ``rich line'') must contain at least $n^2/(64m^2)$ points. 
\end{proof}

\begin{figure}[H] \centering 
    \includegraphics[height=80px]{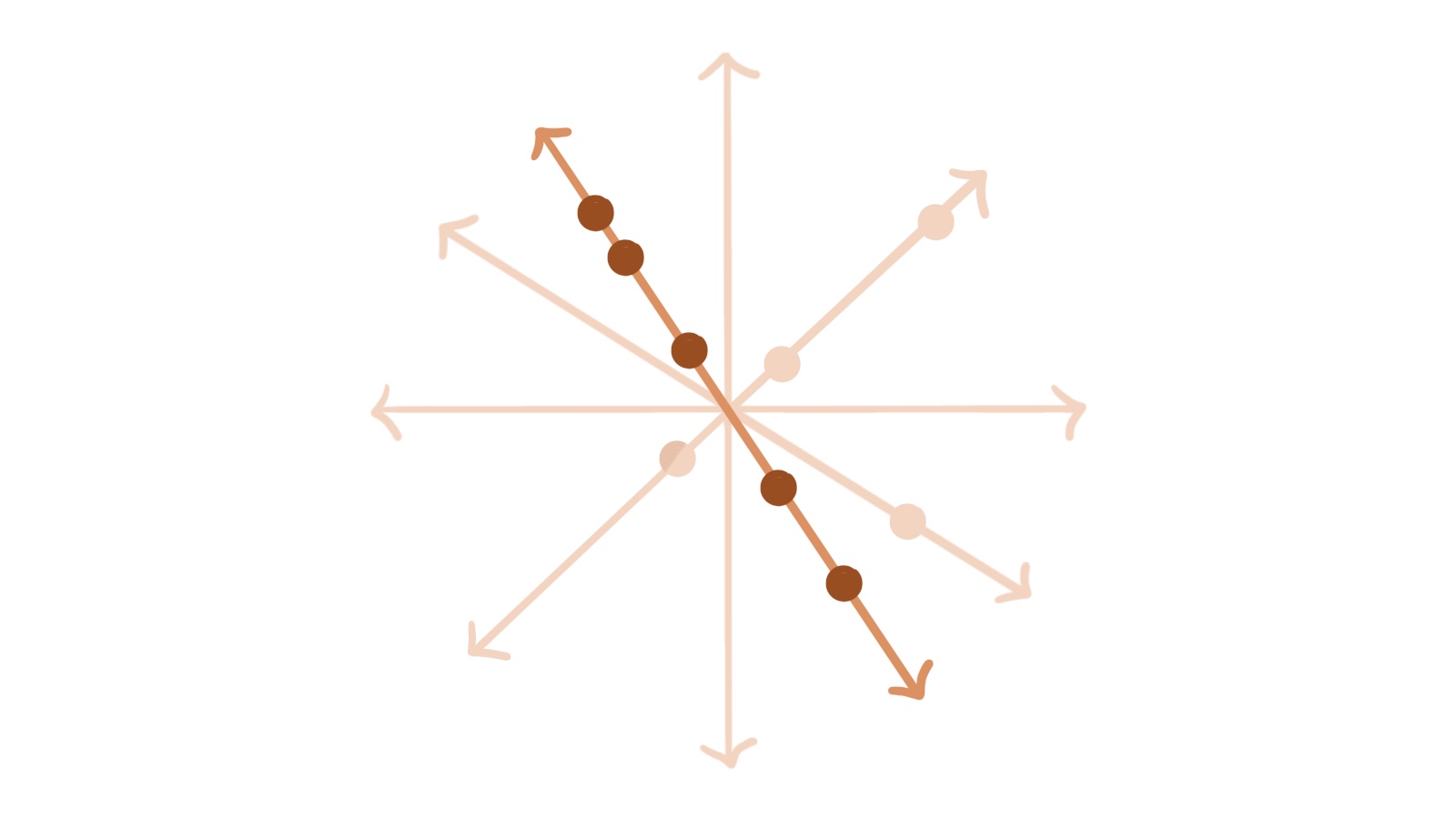} 
    \caption{If an $n$-point configuration has at most $m$ dot products, there exists some rich line with at least $n^2/(64m^2)$ points.} 
\end{figure}

\subsection{Existence of a rich circle}\label{circle-configurations}

We now turn our attention to cocircular points. Similarly to how we dealt with the case of points along a line, we will show that if the distinct dot product set is small, the point configuration must have a ``rich circle.'' Eventually, we will use the existence of this rich circle (in combination with the existence of a rich line) to construct many dot products, leading to our main result.

\begin{lemma}[Bounding dot products on a circle] 
    If we have a set $C \subset \mathbb{R}^2$ consisting of $n$ points \emph{placed anywhere} on a circle centered at the origin, then they determine at least $n/2$ distinct dot products. 
\end{lemma}

\begin{proof} 
    Given a point $p$ on the circle, at most two other points on the circle determine the same dot product with $p$.\footnote{For a fixed dot product $d$, the line $p_x x + p_y y = d$ contains all points $(x,y)$ that determine the dot product $d$ with $p=(p_x, p_y)$, and this line can intersect a circle at most twice.}  Thus, the set of dot products determined by $p$ alone has size at least $n/2$, and thus the size of the total set of dot products must be at least as big.
\end{proof}

Note that the above lemma is nearly sharp, since $n$ points equally spaced on a circle determine $\lfloor n/2\rfloor+1$ dot products.  For the purposes of this paper, this bound will suffice.

Next, we can bound the number of supporting circles of a configuration in a manner very similar to how we bound the number of supporting lines.

\begin{defn}[Supporting circles]
    Given a set of points $P \subset \mathbb{R}^2$, let their \textit{supporting circles} be the circles centered at the origin with radii taken from the set $\set{\sqrt{p_x^2 + p_y^2} \mid (p_x, p_y)  \in P}$.
\end{defn}

\begin{figure}[H] \centering 
    \includegraphics[height=80px]{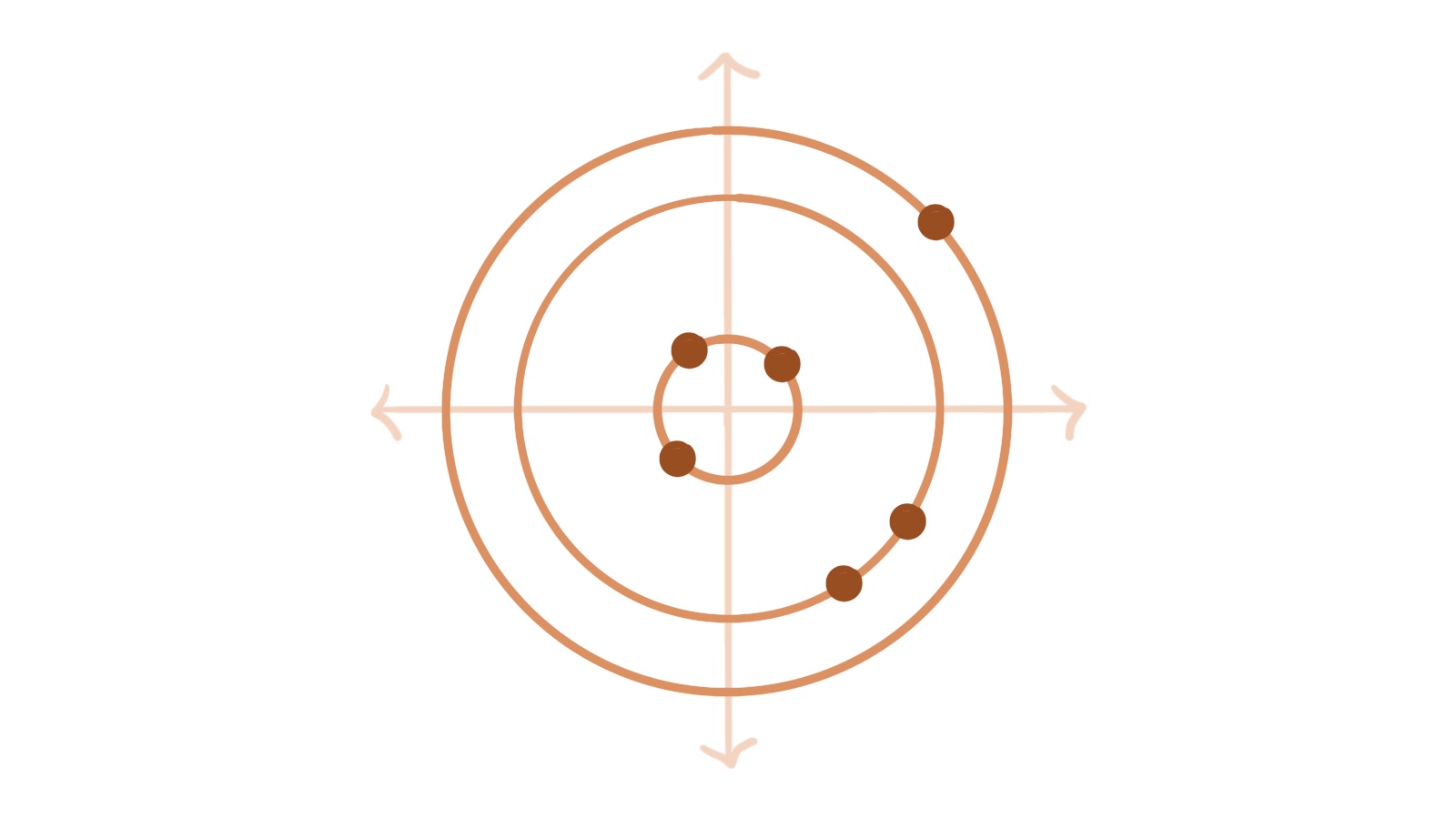} \caption{A set of points and their supporting circles.}
\end{figure}

The next result shows that in any point configuration with few dot products, the number of supporting circles in the configuration must grow with $n$.

\begin{lemma}[Lower bound on \# of supporting circles]
    Let $P_n$ be a set of $n$ points in $\R^2$, with at most $m$ dot products. Then, $P_n$ has at least $n^2/(128m^2)$ supporting circles.  \end{lemma}

\begin{proof}
    We know there must be a rich line with at least $n^2/(64m^2)$ points by Lemma~\ref{rich-line}, and so we have at least half as many supporting circles (since at most two points on the line may lie on the same circle). $\square$
\end{proof} 

Our next result shows that upper-bounding the number of dot products produced by a point configuration also upper-bounds the number of supporting circles in the configuration.

\begin{lemma}[Upper bound on \# of supporting circles] \label{upper-bound-circles}
    Let $P_n$ be a set of $n$ points in $\R^2$, with at most $m$ dot products. Then, $P_n$ has at most $m$ supporting circles. 
\end{lemma}

\begin{proof}
    Take one point $p$ from each supporting circle, and note that the dot product of that point with itself $p \cdot p$ must equal the squared radius of its supporting circle. Thus, the number of distinct dot products is at least the number of supporting circles, and the result follows. 
\end{proof}

Intuitively, we are saying that if a configuration has not too many dot products, then it also has not too many supporting circles. It should therefore follow that there exists one circle which supports many points.

\begin{lemma}[Existence of a rich circle] \label{rich-circle}
    Let $P_n$ be a set of $n$ points in $\R^2$ with at most $m$ dot products. Then, $P_n$ must contain a supporting circle with $n/m$ points --- we call this the ``rich circle.''  
\end{lemma}

\begin{proof}
    By the preceding Lemma~\ref{upper-bound-circles}, there must be at most $m$ supporting circles. Thus, by an averaging argument, at least one of these circles (which we call the ``rich circle'') must contain at least $n/m$ points.
\end{proof}

\begin{figure}[H] \centering 
    \includegraphics[height=80px]{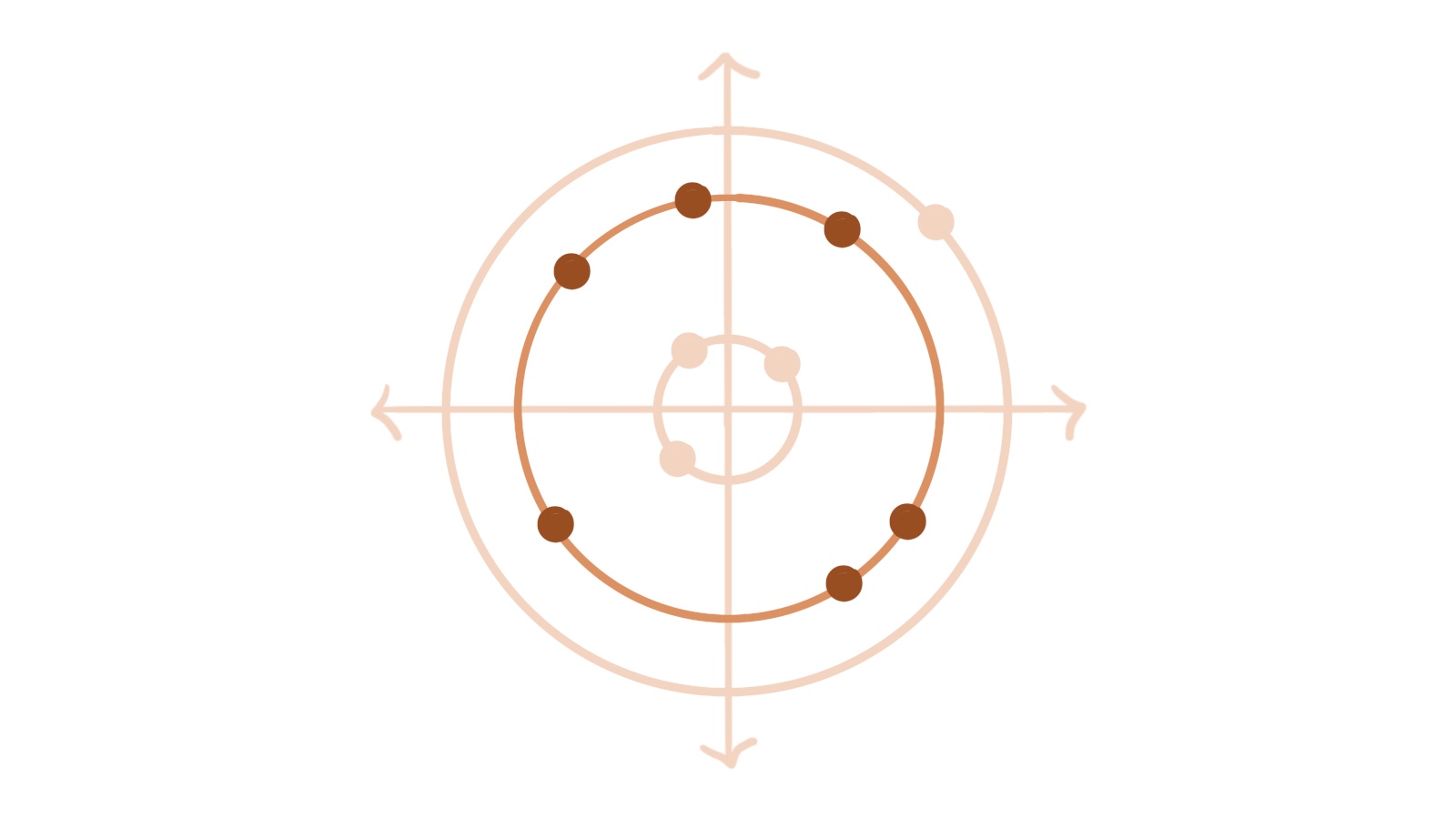} 
    \caption{If an $n$-point configuration has at most $m$ dot products, there exists some rich circle with at least $n/m$ points.} 
\end{figure}

\section{Dot products between a line and circle}\label{line-and-circle-configurations}

As we have shown, a point configuration with few dot products must have both a rich circle and a rich line. It turns out that counting just the dot products between this circle and this line (if the points are well-spaced enough) will yield enough dot products for our bound.

\subsection{Existence of many dot products between a well-spaced  line and circle}\label{bounding-dot-products-between-a-rich-line-and-rich-circle}

To prove the following lemma, we first establish the definition of a ``well-spaced" set of points along a line.

\begin{defn}[Well-Spaced Set]
    Let $L$ be a set of points collinear along a line through the origin. Let us say that $L$ is \emph{well-spaced} with respect to some $b \in (0,1)$ if every consecutive pair of points $\ell_i,\ell_{i+1} \in L$ (where $|\ell_i|<|\ell_{i+1}|$) satisfies $|\ell_i|/|\ell_{i+1}| \leq b$.
\end{defn}

\begin{lemma}[Existence of many dot products between a crowded circle and well-spaced line] \label{circle-and-line}

    Let $b \in (0,1)$. Let $P$ be a set of $n$ points in $\R^2$ consisting of:
    
    \begin{itemize}
        \item (A well-spaced line). A set $L$ of points on a radial line which is well-spaced with respect to $b$.
        \item (A crowded circle). A set $C$ of points on a circle centered at the origin, where the angle $\theta$ between every point in $C$ and the line containing $L$ satisfies $0 \leq \theta \leq \cos^{-1}(b)$.
    \end{itemize}
    
    Then, the number of dot products $|D(P)|$ satisfies $|D(P)| \geq \frac{1}{4} |C||L|$.
\end{lemma}  

\begin{proof}

    Scaling and rotating a configuration of points does not change the number of distinct dot products. Thus, we can assume all points in $C$ lie on the unit circle and all points in $L$ lie on the $x$-axis.

    If the $x$-axis partitions the points in $C$, it suffices to consider only the points in $C$ on one side of the $x$-axis -- specifically the side containing the majority of the points. By an averaging argument, this ``rich part'' of the circle contains at least $|C|/2$ points. Without loss of generality, assume these points are in the upper half-plane. Similarly, we consider only the ``rich ray'' from the origin of the line $L$, which contains at least $|L|/2$ points.\footnote{Note that the choice of ray will change the sign of dot products, but since either all dot products will be negative, or all will be positive, it does not change the number of distinct dot products.  Thus we assume in the rest of the proof, without loss of generality, that these dot products are positive.}

    Fix a point $\ell_i \in L$. Note that since every point in $C$ has an angle $0 \leq \theta \leq \cos^{-1}(b)$, the cosine values satisfy $b \leq \cos \theta \leq 1$. Consequently, all dot products formed by $\ell_i$ and the set $C$ fall within the interval $I_i = [|\ell_i|b, |\ell_i|].$  So we have:

    $$\ell_i \cdot C := \set{\ell_i \cdot c \mid c \in C} \subseteq [|\ell_i|b, |\ell_i|].$$

    We now observe that these intervals are disjoint for distinct points $l_i \neq l_j \in L$. Consider two points $\ell_i, \ell_j \in L$ with $|\ell_i| < |\ell_j|$. By the well-spaced condition, we know that $|\ell_{i}| < b|\ell_j|$. This implies:

    $$ \max(I_i) = |\ell_i| < b|\ell_j| = \min(I_j),$$

    and therefore any two distinct intervals are disjoint.

    There is a geometric way to see this disjointness as well.  Note that $\ell \cdot c = Re(\ell \bar{c})$.  That is, the dot products between points in $L$ and points in $C$ are given by the real parts of the set of complex numbers $\set{\ell \bar{c} | \ell \in L, c \in C}$, which we shall refer to as the set of ``complex dot products" $L \star C$.  We can note that the complex dot products between a fixed point in $\ell \in L$ and points in $C$ are given by points with radius $|\ell|$ and angles in $C$.  Therefore, for $\ell_i \neq \ell_j$, the real projections of $\ell_i \star C$ are disjoint from $\ell_j \star C$.

    Thus, the total number of distinct dot products is at least the number of points in $L$ on the `popular' ray from the origin, times the number of points in $C$ on the `popular' side of $L$:
    
    $$|D(P)|  \geq |\set{\ell \cdot c \mid \ell \in L, c \in C} | \geq \left(\frac{|L|}{2}\right) \left(\frac{|C|}{2}\right) = \frac{1}{4} |C||L|.$$
\end{proof}

We know that in any point configuration with few dot products, there exists a rich line by Lemma~\ref{rich-line} and a rich circle by Lemma~\ref{rich-circle}. We will now apply the preceding Lemma \ref{circle-and-line} to this rich line and rich circle to prove the existence of many dot products under certain conditions.  In particular, as long as all sufficiently rich lines are well-spaced, we will end up with many dot products.

\begin{lemma}[A line-spacing condition for many dot products]\label{line-spacing-for-many-dot-products} 
    Fix $\epsilon \in (0,1]$ and $\delta \in [0,1]$. Let $P$ be a set of $n$ points in $\R^2$, for $n$ sufficiently large.    Suppose that every set of points $L \subseteq P$ along a line through the origin with $|L| \geq 2^{-2/3} n^{\frac{1}{2}(1-\epsilon+3\delta)}$ has a subset $L' \subseteq L$ such that  $|L'|\geq n^{-3\delta} |L|$ and $L'$ is well-spaced with respect to $b$ where $b=\cos(2^{11/3}\pi n^{-\epsilon})$.  Then, $|D(P)| \geq   n^{\frac{3}{4}(1-\epsilon-\delta)}$.
 \end{lemma}

\begin{proof}
    
    Suppose the configuration has strictly fewer than $ n^{\frac{3}{4}(1-\epsilon-\delta)}$ dot products.
    
    First, a sketch of the proof: We strategically choose a densely populated and wedge-shaped subsection of the plane from which to select a rich circle and rich line. The angle of this wedge is exactly $\cos^{-1} (b)$, and therefore, the rich circle must be a `crowded' circle.\footnote{Note that we need $n$ sufficiently large here so that $b=\cos(2^{11/3}\pi n^{-\epsilon})$ is well-defined.}  By considering this crowded circle along with its nearby rich line (which is well-spaced by assumption), we get many dot products by the preceding Lemma~\ref{circle-and-line}. The full argument is as follows.
    
    Let $P_\theta \subseteq P$ consist of points contained in one wedge of the plane, in particular: $$P_\theta := \set{p \in P | \theta \leq \arg p \leq \theta + \cos^{-1}(b)}.$$ 
    
    \begin{figure}[H] \centering 
        \includegraphics[height=150px]{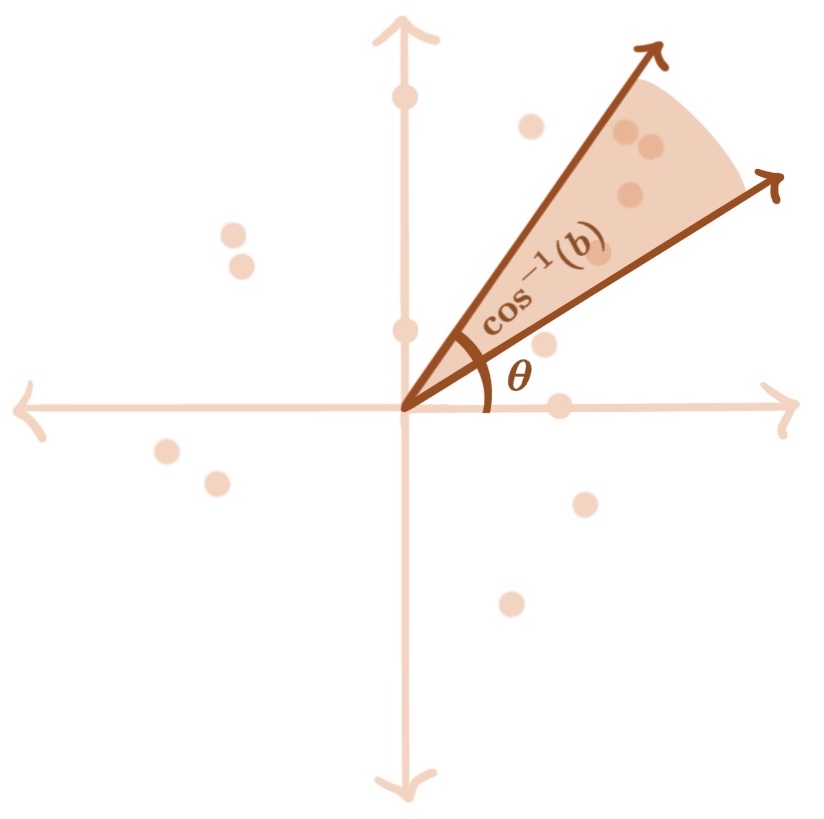}
        \caption{$P_\theta$ contains points from $P$ in one wedge of the plane.}
    \end{figure}
    
    Let $P_{\max}$ be a largest such set $P_\theta$, i.e., $$|P_{\max}| := \max_{\theta \in [0,2\pi)} |P_\theta|.$$
    
    Note that by an averaging argument $|P_{\max}| \geq \left(\frac{\cos^{-1}(b)}{2\pi}\right) n = 2^{8/3} n^{1-\epsilon}$.
    
    In this set of $2^{8/3} n^{1-\epsilon}$ points with fewer than $ n^{\frac{3}{4}(1-\epsilon-\delta)}$ dot products, we must have,  by Lemma \ref{rich-line} (existence of a rich line), a set $L$ of points along a line where
        \[
        |L| \geq \frac{(2^{8/3} n^{1-\epsilon})^2}{64\left( n^{\frac{3}{4}(1-\epsilon-\delta)}\right)^2} = 2^{-2/3}n^{\frac{1}{2}(1-\epsilon+3\delta)},
        \]
    and, by Lemma~\ref{rich-circle} (existence of a rich circle), a set $C$ of points on a circle where
        \[
        |C| \geq \frac{2^{8/3} n^{1-\epsilon}}{ n^{\frac{3}{4}(1-\epsilon-\delta)}}  = 2^{8/3} n^{\frac{1}{4}(1-\epsilon+3\delta)}.
        \]
        
    By assumption, there is some subset $L' \subseteq L$  with $|L'| \geq n^{-3\delta}|L|$ such that $C$ and $L'$ satisfy the conditions in the preceding Lemma~\ref{circle-and-line}, so we have
    \begin{equation*}
        |D(P)| \geq \frac{1}{4} |C| \left(n^{-3\delta}|L|\right) = n^{\frac{3}{4}(1-\epsilon-\delta)},
    \end{equation*}
    which contradicts the starting assumption that $|D(P)| <  n^{\frac{3}{4}(1-\epsilon-\delta)}$.
\end{proof}

\begin{remark}
    Note that the $3/4$ exponent is the best exponent this argument can yield. If we try to go through the above argument with a general $n^\alpha$, we have: Suppose $D(P_n) < c n^{\alpha}$ for some constant $c$.   Then, there must exist a supporting line with at least some constant times $n^{2-2\alpha}$  points--- we call this the ``rich line.'' And there must exist a   supporting circle with some constant times $n^{1-\alpha}$ points --- we call this   the ``rich circle.'' By Lemma~\ref{circle-and-line}, the configuration must   yield some constant times $ (n^{2-2\alpha}) (n^{1-\alpha})$ dot products i.e. $K n^{3-3\alpha}$ dot products for some  constant $K$. That is, we assumed $|D(P_n)| < c n^{\alpha}$, but found   $|D(P_n)| \geq Kn^{3-3\alpha}$, which implies  $\alpha \geq 3/4$.

\end{remark}

\subsection{A structural condition on a rich line}\label{a-density-condition-on-line}

A modified contrapositive of the above theorem shows that we can always find a rich and clustered line in any configuration with few dot products.

\begin{lemma}[Existence of a clustered line] \label{one-clustered-line} 
    Fix $\epsilon \in (0,1]$ and $\delta \in [0,1]$. Let $P$ be a set of $n$ points in $\mathbb{R}^2$, with $n$ sufficiently large, where $|D(P)| < n^{\frac{3}{4}(1-\epsilon-\delta)}$. Then there exists a set of points  $L \subseteq P$ along a line through the origin with $|L| \geq k_1  n^{\frac{1}{2}(1-\epsilon+3\delta)}$ such that at least $\floor{(1-n^{-3\delta})|L|}$ consecutive pairs of points along the line are $(1-k_2n^{-2\epsilon})$-close where $k_1 = 2^{-2/3}$ and $k_2 = 2^{19/3}\pi^2$.
\end{lemma}

\begin{proof}
    The contrapositive of the preceding Lemma~\ref{line-spacing-for-many-dot-products} says the following: for any $\epsilon, \delta \in (0,1]$, if $|D(P)| <  n^{\frac{3}{4}(1-\epsilon-\delta)}$, then there exists a set of collinear points $L \subseteq P$ with $|L|\geq 2^{-2/3} n^{\frac{1}{2}(1-\epsilon+3\delta)}$ where all subsets $L' \subseteq L$ which are well-spaced with respect to $b=1-k_2n^{-2\epsilon}$ have size $|L'|<n^{-3\delta}|L|$. (Note that since $\cos x \geq 1 - \frac{x^2}{2}$ for all real $x$, any pair that is well-spaced with respect to $\cos(2^{11/3}\pi n^{-\epsilon})$ is also well-spaced with respect to $1-k_2n^{-2\epsilon}$.)
    
    Suppose there are $p$ consecutive pairs that are $b$-close.  Then, since there are $|L|-1$ total consecutive pairs, there are $(|L|-1)-p$  consecutive pairs that are not $b$-close.  Choosing just the point with the smaller modulus from each of these $(|L|-1)-p$ pairs will yield a set of points that is well-spaced with respect to $b$.  Since any such set has size strictly less than $n^{-3\delta} |L|$, we know $(|L|-1)-p < n^{-3\delta} |L|$, and therefore $p > (1-n^{-3\delta})|L|-1$.  Since $p$ is an integer, the result follows.
\end{proof}

The preceding lemma can be strengthened by iterating the argument.  We showed that any point configuration with few dot products must contain one rich clustered line.  However, after finding such a line, we may remove its points from the configuration and repeat the argument on the remaining set, implying the set can be partitioned into many such rich clustered lines.  This argument yields our main result.

\medskip

\begin{thmn}[\ref{many-clustered-lines}]{\normalfont  
    (Existence of many clustered lines).} Fix $\epsilon \in (0,1]$ and $\delta \in [0,1]$. Let $P$ be a set of $n$ points in $\mathbb{R}^2$, with $n$ sufficiently large, where $|D(P)| < \left( \frac n2 \right)^{\frac34(1-\epsilon-\delta)}$. Then $P$ contains a subset $P'\subseteq P$ with $|P'|\ge  \frac n2 $ which can be partitioned into disjoint sets $P'=L_1\cup\cdots\cup L_k$, each contained in a radial line, such that $|L_i|\ge k_1 \left( \frac n2 \right)^{\frac12(1-\epsilon+3\delta)}$ and at least $\floor{(1-\left(\frac n2\right)^{-3\delta})|L_i|}$ consecutive pairs of points in $L_i$ are $(1-k_2 \left( \frac n2 \right)^{-2 \epsilon})$-close, where $k_1 = 2^{-2/3}$ and $k_2 = 2^{19/3}\pi^2$.
\end{thmn}
\begin{proof}
    Suppose that $P$ is a set with $|D(P)| <  \left( \frac n2 \right)^{\frac{3}{4}(1-\epsilon-\delta)}$. By Lemma~\ref{one-clustered-line}, there exists a clustered line $L_1 \subseteq P$ containing at least $k_1\left( \frac n2 \right)^{\frac12(1-\epsilon+3\delta)}$ points. Remove $L_1$ from the configuration and consider the remaining set.  If the remaining set still contains at least $n/2$ points, then the same argument applies again, yielding another rich clustered line $L_2$. By continuing this process until the remaining set has less than $n/2$ points, we eventually obtain a collection of disjoint clustered lines $L_1,L_2,\dots,L_k$.
\end{proof}

Thus, any point configuration with few dot products contains a large subset which can be partitioned into rich clustered lines.  

\section{Conclusion}\label{conclusion-a-possible-sketch-of-future-work}

In this paper, we showed that in point configurations with certain spacing conditions (dependent on parameters $\epsilon$ and $\delta$), just counting the dot products between a rich line and a `nearby' rich circle is enough to force  $|D(P)| \geq n^{\frac{3}{4}(1-\epsilon-\delta)}$. 

However, a natural roadblock to proving an $n^{3/4}$ bound in full generality by counting contributions from one rich line and one rich circle is the fact that it is possible to arrange $n^{1/2}$ points on a rich line and $n^{1/4}$ points on a rich circle in a way that only yields $\Theta(n^{1/2})$ dot products.\footnote{For example, consider the sets $L:= \set{(2^i,0) : i \in [\lfloor{n^{1/2}}\rfloor]}$ and $C := \set{e^{i\cos^{-1}(1/2^i)} : i \in [n^{1/4}]}$, and note that they determine few dot products i.e. $|\set{c \cdot \ell : c \in C, \ell \in L}| = \Theta(n^{1/2})$. However, in Lemma \ref{line-spacing-for-many-dot-products}, we choose points along a rich circle and rich line which all belong to one narrow wedge of the plane (in particular, a wedge whose angle is related to the spacing of the rich line, which in this case would be the angle $\cos^{-1}(\frac{1}{2}) = 60^\circ$), and this set of points does not meet this `fits in a narrow wedge' requirement.} Still, this condition is highly constraining -- it seems unlikely that \textit{every} rich line and rich circle in a configuration would determine this few dot products.  

Future work on this problem may involve showing that, in point configurations with few dot products, while it is not true that \emph{every} rich circle and rich line yields $\Omega(n^{3/4})$ dot products, there may \emph{exist} a rich circle and rich line in any such configuration which do determine $\Omega(n^{3/4})$ dot products.  Alternatively, rather than relying on a single rich line–rich circle interaction, one could instead use the fact that configurations with few dot products contain \textit{many} clustered lines (as stated in Theorem \ref{many-clustered-lines}), and attempt to improve the existing lower bound by counting dot product contributions across this entire family of rich, clustered lines.

\section*{Acknowledgments}

I thank Veronica Bitonti, Benjamin Friedman, Ritesh Goenka, Timothy Gowers, Kenneth Moore, A.Y. Odedeyi, and Rikhav Shah for valuable feedback.

\bibliography{refs}{} \bibliographystyle{plain}

\end{document}